\documentclass{amsart}

\usepackage{amssymb}

\usepackage{amsmath}
\usepackage{amsthm}
\usepackage{enumerate}
\usepackage{algorithm,algpseudocode}
\usepackage{color}

\newcommand{\cM}{\mathcal{M}\textnormal{f}}
\newcommand{\cN}{\mathcal{N}}

\newcommand{\St}{\mathcal{S}t_h}
\newcommand{\Af}{\mathbb{A}}

\theoremstyle{plain}
\newtheorem{theorem}{Theorem}[section]
\newtheorem{corollary}[theorem]{Corollary}
\newtheorem{proposition}[theorem]{Proposition}
\newtheorem{lemma}[theorem]{Lemma}

\theoremstyle{definition}
\newtheorem{definition}[theorem]{Definition}

\theoremstyle{remark}
\newtheorem{remark}[theorem]{Remark}
\newtheorem{example}[theorem]{Example}

\begin{document}

\title{Flat families by strongly stable ideals and a generalization of Gr\"obner bases}

\author[F.Cioffi]{Francesca Cioffi}
\address{Dipartimento di Matematica e Applicazioni dell'Universit\`{a} di Napoli Federico II,\\   via Cintia, 80126  Napoli, Italy
         }

\author[M.Roggero]{Margherita Roggero}
\address{Dipartimento di Matematica dell'Universit\`{a} di Torino\\ 
         Via Carlo Alberto 10, 
         10123 Torino, Italy}


\begin{abstract}
Let $J$ be a strongly stable monomial ideal in $S=K[x_0,\ldots,x_n]$ and let $\cM(J)$ be the family of all homogeneous ideals $I$ in $S$ such that the set of all terms outside $J$ is a $K$-vector basis of the quotient $S/I$. We show that an ideal $I$ belongs to $\cM(J)$ if and only if it is generated by a special set of polynomials, the $J$-marked basis of $I$, that in some sense generalizes the notion of reduced Gr\"obner basis and its constructive capabilities. Indeed, although not every $J$-marked basis is a Gr\"obner basis with respect to some term order, a sort of reduced form modulo $I\in \cM(J)$ can be computed for every homogeneous polynomial, so that a $J$-marked basis can be characterized by a Buchberger-like criterion.
Using $J$-marked bases, we prove that the family $\cM(J)$ can be endowed, in a very natural way, with a structure of affine scheme that turns out to be homogeneous with respect to a non-standard grading and flat in the origin (the point corresponding to $J$), thanks to properties of $J$-marked bases analogous to those of Gr\"obner bases about syzygies. 
\end{abstract}

\maketitle




\section*{Introduction}

Let $J$ be any monomial ideal in the polynomial ring $S:=K[x_0,\ldots,x_n]$ in $n+1$ variables such that $x_0 < x_1 < \ldots < x_n$ and let us denote by $\cN(J)$ the set of terms outside $J$. In this paper we consider the family $\cM(J)$ of ideals $I$ of $S$ such that $S=I\oplus \langle \cN (J) \rangle$ as a $K$-vector space and investigate under which conditions this family is in some natural way an algebraic scheme. If $\cN(J)$ is not finite, the family of such ideals can be too large. For instance, if $J=(x_0)\subset K[x_0,x_1]$, the family of all ideals $I$ such that $S/I$ is generated by $\cN(J)=\{x_1^n : n\in \mathbb N\}$ depends on infinitely many parameters because the set $\cN(J)$ has infinite cardinality. Thus, we restrict ourselves to the homogeneous case, so that for every degree $d$ the factor $S_d/I_d\cong (S/I)_d$ is a vector space of finite dimension.  

To study the family $\cM(J)$ we introduce a set of particular homogeneous polynomials, called $J$-marked set, that becomes a $J$-marked basis when it generates an ideal $I$ that belongs to $\cM(J)$. If $J$ is strongly stable, a $J$-marked basis satisfies most of the good properties of a reduced homogeneous Gr\"obner basis and, for this reason, we assume that $J$ is strongly stable. However, even under this assumption, a $J$-marked basis does not need to be a Gr\"obner basis (Example \ref{noGBasis}). We show that a suitable rewriting procedure allows us to compute a sort of reduced forms and to recognize a $J$-marked basis by a Buchberger-like criterion. This criterion is the tool by which we construct the family $\cM(J)$ following the line of the computation of a Gr\"obner stratum, that is the family of all ideals that have $J$ as initial ideal with respect to a fixed term order. In the last years, several authors have been working on Gr\"obner strata, proving that they have!
  a natural and well defined structure of algebraic schemes, that results from a procedure based on Buchberger's algorithm \cite{CF,LR,NS,Ro,RT}, and that they are homogeneous with respect to a non standard positive grading over $\mathbb Z^{n+1}$ \cite{FR}. In this context, it is worth also to recall that \cite{LY} describe a method to compute all liftings of a homogeneous ideal with an approach different from, but close to the method applied to study Gr\"obner strata. 

The paper is organized in the following way. In section 1 we give definitions and basic properties of $J$-marked sets and bases, with several examples. In section 2, under the hypothesis that $J$ is strongly stable, we prove the existence of a sort of reduced form, modulo the ideal generated by a $J$-marked set, for every homogeneous polynomial (Theorem \ref{$J$-normal form}). A consequence is that, if $J$ is strongly stable, a $J$-marked set $G$ is a $J$-marked basis if and only if $J$ and the ideal generated by $G$ share the same Hilbert function (Corollaries \ref{generatore} and \ref{characterization}). From now we suppose that $J$ is strongly stable and in section 3 define a total order (Definitions \ref{order1} and \ref{order2}) on some special polynomials and give an algorithm to compute our reduced forms by a rewriting procedure. This computation opens the access to effective methods for $J$-marked bases, such as a Buchberger-like criterion (Theorem \ref{Buchberger}) !
 that recognizes when a $J$-marked set is a $J$-marked basis $G$, also allowing to lift syzygies of $J$ to syzygies of $G$. 

In section 4 we study the family $\cM(J)$, computing it by the Buchberger-like criterion and showing that there is a bijective correspondence between the ideals of $\cM(J)$ and the points of an affine scheme (Theorem \ref{affine scheme}). A possible objection to our construction is that it depends on a procedure of reduction, which is not unique in general. For this reason we show that $\cM(J)$ has a structure of an affine scheme, that is given by the ideal generated by minors of some matrices and that is homogeneous with respect to a non-standard grading over the additive group $\mathbb Z^{n+1}$ (Lemma \ref{idealeU} and Theorem \ref{homogeneous}). Moreover, we note that $\cM(J)$ is flat in $J$ and that the Castelnuovo-Mumford regularity of every ideal $I\in\cM(J)$ is bounded from above by the Castelnuovo-Mumford regularity of $J$ (Proposition \ref{flat}). In the Appendix, over a field $K$ of characteristic zero, we give an explicit computation of a family $\cM(J)$ which is !
 scheme-theoretically isomorphic to a locally closed subset of the Hilbert scheme of $8$ points in $\mathbb P^2$ (see also \cite{BLR}). We note that it strictly contains the union of all Gr\"obner strata with $J$ as initial ideal and that it is not isomorphic to an affine space, even though the point corresponding to $J$ is smooth.

We refer to \cite{Bu,KR,MM,SPES2} for definitions and results about Gr\"obner bases, in particular to \cite{Mo,Sc} for the approach we follow, and to \cite{Va} for definitions and results about Hilbert functions of standard graded algebras.

A preliminary version of this paper has been written and posed at arXiv:1005.0457 by the second author.


\section{Generators of a quotient $S/I$ and generators of $I$}

In this section we investigate relations among generators of a homogeneous ideal $I$ of $S$ and generators of the quotient $S/I$, under some fixed conditions on generators of $S/I$. 

For every integer $m\geq 0$, the $K$-vector space of all homogeneous polynomials of degree $m$ of $I$ is denoted by $I_m$. The {\em initial degree} of an ideal $I$ is the integer $\alpha_I:=\min\{m\in \mathbb N : I_m\not=0\}$.

We will denote by $x^\alpha = x_0^{\alpha_0} \ldots x_n^{\alpha_n}$ any term in $S$, $\vert\alpha\vert$ is its degree, and we say that $x^\alpha$ divides $x^\beta$ (for short $x^\alpha \vert x^\beta$) if there exists a term $x^\gamma$ such that $x^\beta = x^\alpha x^\gamma$. For every term $x^\alpha\not= 1$ we set $\min(x^\alpha)= \min\{x_i : x_i \vert x^\alpha\}$ and $\max(x^\alpha)= \max\{x_i : x_i \vert x^\alpha\}$.

\begin{definition}\label{def:supp}
The support $Supp(h)$ of a polynomial $h$ is the set of terms that occur in $h$ with non-zero coefficients. 
\end{definition}

If $J$ is a monomial ideal, $B_J$ denotes its (minimal) monomial basis and $\cN (J)$ its \emph{sous-escalier}, that is the set of terms outside $J$. For every polynomial $f$ of $J$, we get $Supp(f)\cap \cN (J)=\emptyset$. 

\begin{definition}\label{def:normalform}
Given a monomial ideal $J$ and an ideal $I$, a {\em $J$-reduced form modulo $I$} of a polynomial $h$ is a polynomial $h_0$ such that $h-h_0\in I$ and $Supp(h_0) \subseteq \cN (J)$. 
\end{definition}

If $I$ is homogeneous, the $J$-reduced form modulo $I$ of a homogeneous polynomial $h$  is supposed to be homogeneous too.

\begin{definition} \cite{RS}
A {\em marked polynomial} is a polynomial $f\in S$ together with a specified term of $Supp(f)$ that will be called {\em head term of $f$} and denoted by $Ht(f)$. 
\end{definition} 

\begin{definition}\label{def:Jbasis}
A finite set $G$ of homogeneous marked polynomials $f_\alpha=x^\alpha-\sum c_{\alpha\gamma} x^\gamma$, with  $Ht(f_\alpha)=x^\alpha$, is called {\em $J$-marked set} if the head terms $Ht(f_\alpha)$ are pairwise different and form the monomial basis $B_J$ of a monomial ideal $J$ and every $x^\gamma$ belongs to $\cN(J)$, so that  $\vert Supp(f)\cap J \vert =1$. A $J$-marked set $G$ is a  $J$\emph{-marked basis} if $\cN(J)$ is a basis of $S/(G)$ as a $K$-vector space, i.e. $S=(G)\oplus \langle \cN (J) \rangle$ as a $K$-vector space.
\end{definition}

\begin{remark}\label{$J$-marked basis}
The ideal $(G)$ generated by a $J$-marked basis $G$ has the same Hilbert function as $J$, hence $dim_K J_m = dim_K (G)_m$ for every $m\geq 0$, by the definition of $J$-marked basis. 
\end{remark}

\begin{definition}\label{def:Mf}
The family of all homogeneous ideals $I$ such that $\cN(J)$ is a basis of the quotient $S/I$ as a $K$-vector space will be denoted by $\cM(J)$ and called {\em $J$-marked family}. 
\end{definition} 

\begin{remark}\label{rem:first}
(1) If $I$ belongs to $\cM(J)$, then $I$ contains a $J$-marked set.

(2) A $J$-marked family $\cM(J)$ contains every homogeneous ideal having $J$ as initial ideal with respect to some term order, but it can also contain other ideals, as we will see in Example~\ref{noGBasis}. 
\end{remark}

\begin{proposition}\label{first properties}
Let $G$ be a J-marked set. The following facts are equivalent:

(i) $G$ is a J-marked basis;

(ii) the ideal $(G)$ belongs to $\cM(J)$;

(iii) every polynomial $h$ of $S$ has a unique J-reduced form modulo $(G)$.
\end{proposition}

\begin{proof}
This follows by the definition of $J$-marked basis.
\end{proof}

\begin{remark}\label{rem:unique}
A $J$-marked basis is unique for the ideal that it generates, by the unicity of $B_J$ and of the $J$-reduced forms of monomials. So, when the ideal $I$ has a $J$-marked bases $G$, the unique $J$-reduced form modulo $I$ can be also called {\em $J$-normal form modulo $I$}.
\end{remark}

In next examples we will see that not every $J$-marked set $G$ is also a $J$-marked basis, even when $(G)$ and $J$ share the same  Hilbert function. Moreover, it can happen that a $J$-marked set $G$ is not a $J$-marked basis, although there exists an ideal $I$ containing $G$ but not generated by $G$ such that $\cN(J)$ is a $K$-basis for $S/I$. 

\begin{example}\label{es1-2-3} 
(i) In $K[x,y,z]$ let $J=(xy,z^2)$ and $I$ be the ideal generated by $f_1=xy+yz, f_2=z^2+xz$, which form a $J$-marked set. Note that $J$ defines a $0$-dimensional subscheme in $\mathbb P^2$, while $I$ defines a $1$-dimensional subscheme, because it contains the line $x+z=0$. Therefore, $I$ and $J$ do not have the same Hilbert function, so that $\{ f_1,f_2 \}$ is not a $J$-marked basis by Remark~\ref{$J$-marked basis}.

(ii) In $K[x,y,z]$, let $J=(xy,z^2)$ and $I$ be the ideal generated by $g_1=xy+x^2-yz, g_2=z^2+y^2-xz$, which form a $J$-marked set. Note that $J$ and $I$ have the same Hilbert function because they are both complete intersections of two quadrics. However, $\cN (J)$ is not free in $K[x,y,z]/I$ because $zg_1+yg_2=x^2z+y^3 \in I$ is a sum of terms in $\cN (J)$. Hence $\{g_1,g_2\}$ is not a $J$-marked basis.

(iii) In $K[x,y,z]$, let $J=(xy,z^2)$ and $I$ be the ideal generated by $f_1=xy+yz, f_2=z^2+xz, f_3=xyz$. Both $I$ and $J$ define $0$-dimensional subschemes in $\mathbb P^2$ of degree $4$. Moreover, $I$ belongs to $\cM(J)$ because for every $m\geq 2$ the $K$-vector space $U_m=I_m+\cN (J)_m=I_m+\langle x^m,y^m,x^{m-1}z, y^{m-1}z \rangle$ is equal to $K[x,y,z]_m$. This is obvious for $m=2$. Assume $m\geq 3$. Then, $U_m$ contains all the terms $y^{m-i}z^i$, because  $yz^2=zf_1-f_3$ belongs to $I$. Moreover $U_m$ contains all the terms $x^{m-i}y^i$ because $x^2y=xf_1-f_3\in I$ and $xy^{m-1}=y^{m-2}f_1-zy^{m-1}\in U_m$. Finally, by induction on $i$, we can see that all the terms $x^{i}z^{m-i}$ belong to $U_m$. Indeed, as already proved, $z^m$ belongs to $U_m$, hence $x^{i-1}z^{m-i+1}\in U_m$ implies $x^{i}z^{m-i}=x^{i-1}z^{m-i-1}f_2-x^{i-1}z^{m-i+1}\in U_m$. 
However, the $J$-marked set $G=\{f_1,f_2\}$ does not generate $I$ and is not a $J$-marked basis, as shown in (i). 
\end{example}


\section{Strongly stable ideals $J$ and $J$-marked bases}

In this section we show that the properties of $J$-marked sets improve decisively if $J$ is strongly stable. 

Recall that a monomial ideal $J$ is strongly stable if and only if, for every $x_0^{\alpha_0} \ldots x_n^{\alpha_n}$ in $J$, also the term $x_0^{\alpha_0}\ldots x_i^{\alpha_i-1} \ldots x_j^{\alpha_j+1}\ldots x_n^{\alpha_n}$ belongs to $J$, for each $0\leq i <j\leq n$ with $\alpha_i>0$, or, equivalently, for every $x_0^{\beta_0} \ldots x_n^{\beta_n}$ in $\cN(J)$, also the term $x_0^{\beta_0}\ldots x_h^{\beta_h+1} \ldots x_k^{\beta_k-1}\ldots x_n^{\beta_n}$ belongs to $\cN(J)$, for each $0\leq h <k\leq n$ with $\beta_k>0$.

A strongly stable ideal is always Borel-fixed, that is fixed under the action of the Borel subgroup of lower-triangular invertibles matrices. If $ch(K)=0$, also the vice versa holds (e.g. \cite{D}) and \cite{Ga} guarantees that in generic coordinates the initial ideal of an ideal $I$, with respect to a fixed term order, is a constant Borel-fixed monomial ideal, denoted by $gin(I)$ and called the {\em generic initial ideal} of $I$. 

Recall that some Gr\"obner-like bases and their structure were introduced by Janet \cite{J1,J2,P} and the related algorithm has been discussed as an alternative to Buchberger's algorithm under the name of involutive bases by Gerdt and Blinkov \cite{GB1,GB2}. In \cite{Ma} the author investigates interrelation of Borel-fixed ideals and existence (finiteness) of their Pommaret bases. In doing so, a Pommaret basis exists if and only if it is a minimal Janet basis (see \cite{Ge}).

In \cite{RS} a {\em reduction relation $\stackrel{\mathcal F}\longrightarrow$} modulo a given set $\mathcal F$ of marked polynomials is defined in the usual sense of Gr\"obner bases theory and it is proved that, if  $\stackrel{\mathcal F}\longrightarrow$ is Noetherian, then there exists an admissible term order $\prec$ on $S$ such that $Ht(f)$ is the $\prec$-leading term of $f$, for all $f\in \mathcal F$, being the converse already known \cite{Bu}. A similar approach has been proposed in \cite{MaRe1} and better explained in \cite{MaRe2} for defining and computing Gr\"obner bases in group rings.
 
If we take a $J$-marked set $G$, $\stackrel{G}\longrightarrow$ can be non-Noetherian, as the following example shows. However, we will see that, if $J$ is a strongly stable ideal and $G$ is a $J$-marked set, every homogeneous polynomial has a $J$-reduced form modulo $(G)$. 

\begin{example}\label{es-8} Let us consider the $J$-marked set $G=\{f_1=xy+yz, f_2=z^2+xz\}$, where $Ht(f_1)=xy$ and $Ht(f_2)=z^2$. The term $h=xyz$ can be rewritten only by $xyz-zf_1=-yz^2$ and the term $-yz^2$ can be rewritten only by $-yz^2+yf_2=xyz$, which is again the term we wanted to rewrite. Hence, the reduction relation $\stackrel{G}\longrightarrow$ is not Noetherian. Observe that in this case $J=(xy,z^2)$ is not strongly stable, but $\stackrel{G}\longrightarrow$ can be non-Noetherian also if $J$ is strongly stable, as Example \ref{noGBasis} will show.
\end{example}

\begin{theorem}{\rm (Existence of $J$-reduced forms)}\label{$J$-normal form}
Let $G=\{f_\alpha=x^\alpha-\sum c_{\alpha\gamma} x^\gamma : Ht(f_\alpha)=x^\alpha\in B_J\}$ be a $J$-marked set, with $J$ strongly stable. Then, every polynomial of \ $S$ has a $J$-reduced form modulo $(G)$.
\end{theorem}

\begin{proof}
It is sufficient to prove that our assertion holds for the terms, because every polynomial is a linear combination of terms. Let us consider the set $E$ of terms which have not a $J$-reduced form modulo $(G)$. Of course $E\cap B_J=\emptyset$. If $E$ is not empty and $x^\beta$ belongs to $E$, then $x^\beta=x_i x^\delta$ for some $x^\delta$ in $J$. We choose $x^\beta$ so that its degree $m$ is the minimum in $E$ and that, among the terms of degree $m$ in $E$,  $x_i$ is minimal. Let $\sum c_{\delta\gamma} x^\gamma$ be a $J$-reduced form modulo $(G)$ of $x^\delta$, that exists by the minimality of $m$. Thus we can rewrite $x^\beta$ by $\sum c_{\delta\gamma} x_i x^\gamma$. We claim that all terms $x_i x^\gamma$ do not belong to $E$. On the contrary, if $x_ix^\gamma$ belongs to $E$, then $x_ix^\gamma=x_jx^\epsilon$ for some $x^\epsilon$ in $J$. If it were $x_i<x_j$ then, by the strongly stable property and since $x^\gamma$ belongs to $\cN(J)$, we would get that $x^\epsilon = x_i {!
 x^\gamma}/{x_j}$ belongs to $\cN(J)$, that is impossible. So, we have $x_j<x_i$ and by the minimality of $x_i$ the term $x_ix^\gamma$ has a $J$-reduced form modulo $(G)$. This is a contradiction and so $E$ is empty.
\end{proof}

\begin{corollary} \label{generatore}
If $J$ is a strongly stable ideal and $I$ a homogeneous ideal containing a $J$-marked set $G$, then $\cN(J)$ generates $S/I$ as a $K$-vector space. Thus $dim_K I_m \geq dim_K J_m$, for every $m\geq 0$. 
\end{corollary}

\begin{proof}
By Theorem~\ref{$J$-normal form}, for every polynomial $h$ there exists a polynomial $h_0$ such that $h-h_0$ belongs to $(G)\subseteq I$ and $Supp(h_0)\subseteq \cN(J)$. So, all the elements of $S/I$ are linear combinations of terms of $\cN(J)$ and the claim follows.
\end{proof}  

\begin{corollary}\label{characterization}
Let $J$ be a strongly stable ideal and $G$ be a $J$-marked set. Then, $G$ is a $J$-marked basis if and only if $dim_K (G)_m \leq dim_K J_m$, for every $m\geq 0$ or, equivalently, $\cN(J)$ is free in $S/(G)$. 
\end{corollary}

\begin{proof} 
By Proposition \ref{first properties}, $G$ is a $J$-marked basis if and only if every polynomial has a unique $J$-reduced form modulo $(G)$. So, it is enough to apply Theorem~\ref{$J$-normal form} and Corollary~\ref{generatore}.
\end{proof}

\begin{corollary}\label{cor:ultimo}
Let $J$ be a strongly stable ideal and $I$ be a homogeneous ideal. Then
$I$ belongs to $\cM(J)$ if and only if $I$ has a $J$-marked basis.
\end{corollary}

\begin{proof} 
If $I$ has a $J$-marked basis then $I$ belongs to $\cM(J)$ by definition. Vice versa, apply Remark \ref{rem:first}(1) and Corollary \ref{characterization}.
\end{proof}

\begin{remark}\label{generic coordinate}
Every reduced Gr\"obner basis of a homogeneous ideal with respect to a graded term order is a $J$-marked basis for some monomial ideal $J$, hence every homogeneous ideal contains a $J$-marked basis. But, unless we are in generic coordinates, not every (homogeneous) ideal contains a $J$-marked basis with $J$ strongly stable, as for example a monomial ideal which is not strongly stable. 
\end{remark}

Let $G$ be a $J$-marked basis with $J$ strongly stable. Thanks to the existence and the unicity of $J$-reduced forms, $G$ can behave like a Gr\"obner basis in solving problems, as the membership ideal problem in the homogeneous case. Indeed, by the unicity of $J$-reduced forms, a polynomial belongs to the ideal $(G)$ if and only if its $J$-reduced form modulo $(G)$ is null. But, until now, we do not yet have a computational method to construct $J$-reduced forms.  

In next section, by exploiting the proof of Theorem \ref{$J$-normal form}, we  provide an algorithm which, under the hypothesis that $J$ is strongly stable, reduces every homogeneous polynomial to a $J$-reduced form modulo $(G)$ in a finite number of steps, although $\stackrel{G}\longrightarrow$ is not necessarily Noetherian. This fact allows us also to recognize  when a $J$-marked set is a $J$-marked basis by a Buchberger-like criterion and, hence, to develop effective computational aspects of $J$-marked bases.


\section{Effective methods for $J$-marked bases}

Let $I$ be the homogeneous ideal generated by a $J$-marked set $G=\{f_\alpha=x^\alpha-\sum c_{\alpha\gamma} x^\gamma : Ht(f_\alpha)=x^\alpha\in B_J\}$, where {\em $J$ is strongly stable}, so that every polynomial has a $J$-reduced form modulo $I$, by Theorem~\ref{$J$-normal form}. 

In this section we obtain an efficient procedure to compute in a finite number of steps a $J$-reduced form modulo $I$ of every homogeneous polynomial. To this aim, we need some more definitions and results.

For every degree $m$, the $K$-vector space $I_m$ formed by the homogeneous polynomials of degree $m$ of $I$ is generated by the set $W_m=\{x^\delta f_\alpha : x^{\delta+\alpha} \text{ has degree } m, f_\alpha \in G\}$, that becomes a set of marked polynomials by letting $Ht(x^\delta f_\alpha)=x^{\delta+\alpha}$. 

\begin{lemma}\label{minimum}
Let $x^\beta$ be a term of $J_m\setminus B_J$ and $x_i=\min(x^\beta)$. Then ${x^\beta}/{x_i}$ belongs to $J_{m-1}$. 
\end{lemma}

\begin{proof}
By the hypothesis there exists at a least a term of $J_{m-1}$ that divides the given term $x^\beta$. So, let $x_j$ such that $x^\beta/x_j$ belongs to $J_{m-1}$. If $x_j=x_i$, we are done. Otherwise, we get $x^\beta=x_i x_j x^\delta$, for some term $x^\delta$, so that $x_i x^\delta=x^\beta/x_j$ belongs to $J_{m-1}$. By the definition of a strongly stable ideal and since $x_j>x_i$, we obtain that $x^\beta/x_i=x_jx^\delta$ belongs to $J_{m-1}$. 
\end{proof}

The property of Borel ideals, that we point out by Lemma \ref{minimum}, allows us to define the following special subset of $W_m$, by induction on $m$.

\begin{definition}\label{def:$V_m$}
If $m=\alpha_J$ is the initial degree of $J$, we set $V_m:=G_m$; so, for every term $x^\beta\in B_J$ of degree $\alpha_J$, there is a unique polynomial $g_\beta\in V_{\alpha_J}$ such that $Ht(g_\beta)=x^\beta$. If $m=\alpha_J+1$, for every $x^\beta \in J_{\alpha_J+1}\setminus G_{\alpha_J+1}$, we set $g_\beta:=x_i g_\epsilon$, where 
$x_i=\min(x^\beta)$ and $g_\epsilon$ is the unique polynomial of $V_{\alpha_J}$ such that $Ht(g_\epsilon)=x^\epsilon$. Thus, we let $V_{\alpha_J+1}:=G_{\alpha_J+1}\cup \{g_\beta \ : \ x^\beta \in J_{\alpha_J+1}\setminus B_J\}$. Analogously, for every $m>\alpha_J$ and for every $x^\beta\in J_m\setminus B_J$, we set $g_\beta:=x_i g_\epsilon$, where $x_i=\min(x^\beta)$ and $g_\epsilon$ is the unique polynomial of $V_{m-1}$ with head term $x^\epsilon={x^\beta}/{x_i}$, and we let $V_m:=G_m\cup \{g_\beta \ : \ x^\beta \in J_m\setminus B_J\}$.
\end{definition}

\begin{remark}\label{rem:minimum}
By construction, for every element $g_\beta$ of $V_m\subseteq W_m$ there exist $x^\delta$ and $f_\alpha\in G$ such that $g_\beta=x^\delta f_\alpha$ and $x^\delta=1$ or $\max(x^\delta)\leq \min(x^\alpha)$. Indeed, it is enough to take $g_{\beta_1}=g_\beta/\min(x^\beta)$, $g_{\beta_2}=g_{\beta_1}/\min(x^{\beta_1})$ and so on, until we obtain a polynomial $f_\alpha$ of $G$ and the term $x^\delta=\min(x^\beta)\cdot \prod\min(x^{\beta_i)}$. In particular, we get $\min(x^\delta)=\min(x^\beta)$. 
\end{remark}

For every integer $m\geq \alpha_J$, we define the following total order $\succeq_m$ on $V_m$. Note that we start by fixing any ordering on $G_m$, that needs not to be a term order.

\begin{definition}\label{order1}
Let $G_m$ be ordered with respect to any order $\geq$ and, for every $f_\alpha, f_{\alpha'} \in G_m$, let $f_\alpha \succeq_m f_{\alpha'}$ if and only if $f_\alpha \geq f_{\alpha'}$. For every $g_\beta \in V_m\setminus G_m$ and $f_\alpha \in G_m$, we set $g_{\beta} \succeq_m f_\alpha$. For every $m>\alpha_J$, given $x_ig_\epsilon$, $x_jg_{\eta} \in V_m\setminus G_m$, where $x_i=\min(x_ix^\epsilon)$ and $x_j=\min(x_jx^\eta)$, we set 
$$x_ig_\epsilon \succeq _m x_jg_{\eta} \Leftrightarrow x_i > x_j \text{ or } x_i=x_j \text{ and } g_\epsilon \succeq _{m-1} g_{\eta}.
$$
\end{definition}

By the definition of $V_m$ and by well-known properties of a strongly stable ideal, we get the routine \textsc{VConstructor} to compute $V_m$, for every $\alpha_J\leq m \leq s$. 

\begin{algorithm}
\begin{algorithmic}[1]
\Procedure{VConstructor}{$G$,$s$} $\rightarrow V_{\alpha_j}\ldots,V_{s}$
\Require $G$ is a $J$-marked set so that $G_m$ is ordered with respect to any order, for every $m\geq \alpha_J$, with $J$ a strongly stable ideal, and $s\geq \alpha_J$. 
\Ensure $V_m$ ordered by $\succeq_m$, for every $\alpha_J\leq m \leq s$
\State $\alpha_J:=\min\{deg(Ht(f_\alpha)) \vert f_\alpha \in G\}$
\State $V_{\alpha_J}:=G_\alpha$
\For{$m=\alpha_J+1$ to $s$}
\State $V_m:=G_m$;
\For{$i=0$ to $n$}
\For{$j=1$ to $\vert V_{m-1}\vert$}
\If{$i\leq \min(Ht(V_{m-1}[j]))$}
\State $V_m=V_m\cup\{x_iV_{m-1}[j]\}$
\EndIf 
\EndFor
\EndFor
\EndFor
\State \Return $V_{\alpha_j}\ldots,V_{s}$;
\EndProcedure
\end{algorithmic}
\end{algorithm}

\begin{lemma}\label{reduction1}
With the above notation,
$$x_i g_\epsilon \in V_m\setminus G_m \text{ and } x^\beta \in Supp(x_i g_{\epsilon})\setminus\{x_i x^\epsilon\} \text{ with } \ g_\beta\in V_m  \Rightarrow x_i g_{\epsilon} \succ_m g_\beta.$$
\end{lemma}

\begin{proof}
By induction on $m$, first observe that for $m=\alpha_J$ there is nothing to prove because $V_{\alpha_J}=G_{\alpha_J}$. For $m>\alpha_J$, let $g_\beta=x_j g_\eta \not\in G_m$. If $x_i=x_j$, then $x^\eta$ belongs to $Supp(g_\epsilon)\setminus\{x^\epsilon\}$ and, by the induction, we have $g_\eta \prec_{m-1} g_\epsilon$. Otherwise, note that every term of $Supp(x_i g_{\epsilon})$ is divided by $x_i$, so $x_jx^\eta=x_i x^\lambda$ and, by Remark \ref{rem:minimum}, we get $x_j=\min(x^\beta)=\min(x_i x^\lambda) \leq x_i$. 
\end{proof}

\begin{proposition} \label{construction of $J$-normal form} {\rm (Construction of $J$-reduced forms) }
With the above notation, every term $x^\beta\in J_m\setminus G_m$ can be reduced to a $J$-reduced form modulo $I$ in a finite number of reduction steps, using only polynomials of $V_m$. Hence, the reduction relation $\stackrel{V_m}\longrightarrow$ is Noetherian in $S_m$. 
\end{proposition}

\begin{proof}
By definition of $V_m$, every term $x^\beta$ of $J_m$ is the head term of one and only one polynomial $g_\beta$ of $V_m\subseteq W_m$. Hence, we rewrite $x^\beta$ by $g_\beta$ getting a $K$-linear combination of terms belonging to  
$Supp(g_\beta)\setminus \{x^\beta\}$. Applying Lemma~\ref{reduction1} repeately,
we are done since $V_m$ is a finite set.
\end{proof}

\begin{definition} 
A homogeneous polynomial, with support contained in $\cN(J)$ and in relation by $\stackrel{V_m}\longrightarrow$ to a homogeneous polynomial $h$ of degree $m$, is denoted by $\bar h$ and called {\em $V_m$-reduction of $h$}.
\end{definition} 

For every homogeneous polynomial $h$ of degree $m$, $\bar h$ is a $J$-reduced form modulo $I$. Hence, from the procedure described in the proof of Proposition~\ref{construction of $J$-normal form} we obtain the routine \textsc{ReducedFormConstructor} that, actually, forms a step of a division algorithm with respect to a $J$-marked set, with $J$ strongly stable. 

\begin{algorithm}
\begin{algorithmic}[1]
\Procedure{ReducedFormConstructor}{$h$,$V_m$} $\rightarrow \bar h$
\Require $h$ is a homogeneous polynomial of degree $m$
\Require a list $V_m$, as defined in Definition \ref{def:$V_m$}, and ordered by $\succeq_m$
\Ensure $V_m$-reduction $\bar h$ of $h$
\State $L:=\vert V_m\vert$;
\For{$K=1$ to $L$}
\State $x^\eta :=Ht(V_m[K])$; 
\State $a$:=coefficient of $x^\eta$ in $h$;
\If{$a\not=0$}
\State $h:=h-a\cdot V_m[K]$;   
\EndIf;
\EndFor
\State \Return $h$;
\EndProcedure
\end{algorithmic}
\end{algorithm}

\begin{remark}\label{future}
(1) There is a strong analogy between the union of the sets $V_m$ and the so-called {\em staggered bases}, introduced by \cite{GM} and studied also by \cite{MMT}. Moreover, the procedure to construct the sets $V_m$ mimics the one introduced by Janet in a context in which it was assumed (in generic coordinates) that the ideal generated by the head terms is Borel, and thus it is sufficient to extend the basis by multiplying each polynomial $g$ by variables $x_i\leq \min(Ht(g))$. Indeed, for constructing the set $V_m$ we multiply the polynomials of $V_{m-1}$ by the same variables considered for Janet bases. For this reason, we plan to study relations between $J$-marked bases and Janet bases in a future work, in which also a comparison with Border Bases would be interesting because of the lack of a term order \cite{MMM,Mou,MouT1,MouT2}. Anyway, we must point out that in our context the classical problem \lq\lq given a basis of an ideal, extend it for computing a Gr\"obner-like b!
 asis" is unnatural, because we have not an ideal, but we want to construct the family of all the ideals $I$ with a suitable $K$-vector basis of $S/I$.

(2) In the procedure \textsc{ReducedFormConstructor} we reduce a polynomial using $V_m$. Thus, it would be better if $V_m$ consisted of already reduced polynomials, in analogy with well-known efficient algorithms \cite{FGLM,MB}. Anyway, we think that our algorithm can be improved and we are making efforts in this direction.
\end{remark}

Now, we extend to $W_m$ the order $\succeq_m$ defined on $V_m$. In our setting, a term $x^\delta$ is higher than a term $x^{\delta'}$ with respect to the degree reverse lexicographic term order (for short $x^\delta >_{drl} x^{\delta'}$) if $\vert \delta\vert > \vert \delta'\vert$ or if $\vert \delta\vert = \vert \delta'\vert$ and the first non null entry of $\delta-\delta'$ is negative.

\begin{definition}\label{order2}
Let the polynomials of $G_m$ be ordered as in Definition \ref{order1} by any order $\geq$ and $x^\delta f_\alpha$, $x^{\delta'}f_{\alpha'}$ be two elements of $W_m$. We set 
$$x^\delta f_\alpha \succeq_m x^{\delta'}f_{\alpha'} \Leftrightarrow 
x^\delta >_{drl} x^{\delta'} \text{ or } x^\delta = x^{\delta'} \text{ and } f_\alpha\geq f_{\alpha'}.$$
\end{definition}

\begin{lemma}\label{reduction2} 
(i) For every two elements $x^\delta f_\alpha$, $x^{\delta'}f_{\alpha'}$ of $W_m$ we get $$x^\delta f_\alpha \succeq_m x^{\delta'}f_{\alpha'} \Rightarrow \forall x^\eta :\ \  x^{\delta+\eta} {f_\alpha} \succeq_{m'} x^{\delta'+\eta}f_{\alpha'},$$ where $m'=\vert \delta +\eta + \alpha\vert$.

(ii) Every polynomial $g_\beta\in V_m$ is the minimum with respect to $\preceq_m$ of the subset $W_\beta$ of $W_m$ containing all polynomials of $W_m$ with $x^\beta$ as head term.

(iii) $x^\delta f_\alpha \in W_m\setminus G_m \text{ and } x^\beta \in Supp(x^\delta f_\alpha)\setminus\{x^\delta x^\alpha\} \text{ with } \ g_\beta\in V_m  \Rightarrow x^\delta f_{\alpha} \succ_m g_\beta$.
\end{lemma}

\begin{proof}
(i) This follows by the analogous property of the term order $>_{drl}$.

(ii) The statement holds by construction of $V_m$ and by Remark \ref{rem:minimum}. Indeed, by the same arguments as before, if $x^\delta f_\alpha$ is any polynomial of $W_\beta$ and $g_\beta=x^{\delta'}f_{\alpha'}\in V_m$, with $\max(x^{\delta'})\leq \min(x^{\alpha'})$ as in Remark \ref{rem:minimum}, then $x_j=\min(x^{\delta'})=\min(x^{\delta'+\alpha'})=\min(x^{\delta+\alpha})\leq \min(x^\delta)$. If the equality holds, it is enough to observe that $\frac{x^{\delta}}{x_j}f_{\alpha} \in W_{m-1}$ and  $\frac{x^{\delta'}}{x_j}f_{\alpha'} \in V_{m-1}$ by construction.

(iii) The proof is analogous to the proof of Lemma \ref{reduction1}. If $x^\beta$ belongs to  $J_m$ we are done. Otherwise, note that every term of $Supp(x^\delta f_\alpha)$ is a multiple of $x^\delta$, in particular $x^{\delta'+\alpha'}=x^{\delta+\gamma}$ for some $x^\gamma \in \cN(J)$. Let $x_i=\min(x^{\delta})$ and $x_j=\min(x^{\delta'})$. 
By Remark \ref{rem:minimum}, we get $x_j=\min(x^{\delta'+\alpha'})=\min(x^{\delta+\gamma})\leq \min(x^\delta)=x_i$. 
If $x_j=x_i$, then ${x^\beta}/{x_i}$ belongs to the support of $\frac{x^\delta}{x_i}f_\alpha$. Now we use induction. 
\end{proof}

In Remark \ref{generic coordinate} we have already observed that in generic coordinates every homogeneous ideal has a $J$-marked basis, with $J$ strongly stable. Now, given a strongly stable ideal $J$, we describe a {\em Buchberger-like} algorithmic method to check if a $J$-marked set is a $J$-marked basis, recovering the well-known notion of {\em S-polynomial} from the Gr\"obner bases theory.

\begin{definition}
The {\em S-polynomial} of two elements $f_\alpha$, $f_{\alpha'}$ of a $J$-marked set $G$ is the polynomial $S(f_\alpha, f_{\alpha'}):=x^\beta f_\alpha-x^{\beta'}f_{\alpha'}$, where $x^{\beta+\alpha}=x^{\beta'+\alpha'}=lcm(x^\alpha,x^{\alpha'})$.
\end{definition}  

\begin{theorem}{\rm (Buchberger-like criterion)}\label{Buchberger}
Let $J$ be a strongly stable ideal and $I$ the homogeneous ideal generated by a $J$-marked set $G$. With the above notation:
$$I\in \cM(J) \Leftrightarrow \overline{S(f_\alpha, f_{\alpha'})}=0, \forall f_\alpha, f_{\alpha'} \in G.$$
\end{theorem}

\begin{proof}
Recall that $I \in \cM(J)$ if and only if $G$ is a $J$-marked basis, so that every polynomial has a unique $J$-reduced form modulo $I$. Since $S(f_\alpha, f_{\alpha'})$ belongs to $I$ by construction, its $J$-reduced form modulo $I$ is zero and coincides with $\overline{S(f_\alpha, f_{\alpha'})}$, by the unicity of $J$-reduced forms.

For the converse, by Corollary~\ref{characterization} it is enough to show that, for every $m$, the $K$-vector space $I_m$ is generated by the $dim_K J_m$ elements of $V_m$.  More precisely we will show that every polynomial $x^\delta f_\alpha\in W_m$ either belongs to $V_m$ or is a $K$-linear combination of elements of $V_m$ lower than $x^\delta f_\alpha$ itself. We may assume that this fact holds for every polynomial in $W_m$ lower than $x^\delta f_\alpha$. If $x^\delta f_\alpha$ belongs to $V_m$ there is nothing to prove. If $x^\delta f_\alpha$ does not belong to $V_m$, let $x^{\delta'}f_{\alpha'}=\min(W_{\delta+\alpha})\in V_m$, so that $x^\delta f_\alpha \succ_m x^{\delta'} f_{\alpha'}$, and consider the polynomial $g=x^\delta f_\alpha - x^{\delta'} f_{\alpha'}$. 

If $g$ is the $S$-polynomial $S(f_\alpha,f_{\alpha'})$, then it is a $K$-linear combination $\sum c_i g_{\eta_i}$ of polynomials of $V_m$ because $\overline{S(f_\alpha, f_{\alpha'})}=0$ by the hypothesis. Moreover,  by construction, $x^{\delta'} f_{\alpha'}$ belongs to $V_m$ and, thanks to Lemma~\ref{reduction2}(iii), for all $i$ we have $x^\delta f_\alpha\succ_m g_{\eta_i}$.

If $g$ is not the $S$-polynomial $S(f_\alpha,f_{\alpha'})$, then there exists a term $x^\beta\not=1$ such that $g=x^\beta S(f_\alpha,f_{\alpha'})= x^\beta (x^\eta f_\alpha- x^{\eta'}f_{\alpha'})$. By the hypothesis $S(f_\alpha,f_{\alpha'})$ is a $K$-linear combination $\sum c_i g_{\eta_i}$ of elements of $V_{m-\vert \beta\vert}$ lower than $x^\eta f_\alpha$. Hence, $x^\delta f_\alpha= x^{\delta'}f_{\alpha'}+ \sum c_i x^\beta g_{\eta_i}$, where all polynomials appearing in the right hand are lower than $x^\delta f_\alpha$ with respect to $\succ_m$, by Lemma~\ref{reduction2}(i). So we can apply to them the inductive hypothesis for which either they are elements of $V_m$ or they are $K$-linear combinations of lower elements in $V_m$. This allows us to conclude the proof.
\end{proof}

Let $H=(h_1,\ldots,h_t)$ be a syzygy of a $J$-basis $G=\{f_{\alpha_1},\ldots,f_{\alpha_t}\}$ such that every polynomial $h_i=\sum c_{i\beta} x^{\beta}$ is homogeneous and every product $h_i f_{\alpha_i}$ has the same degree $m$.
A syzygy $M=(m_1,\ldots,m_t)$ of $J$ is homogeneous if, for every $1\leq i\leq t$, we have $m_ix^{\alpha_i}=c_{i\epsilon} x^\epsilon$, for a constant term $x^\epsilon$ and $c_{i\epsilon} \in K$. 

\begin{definition}\label{def:lifting}
The {\em head term $Ht(H)$ of the syzygy $H$} is the head term of the polynomial $H_{max}:=$ $\max_{\succeq_m}\{x^\beta f_{\alpha_i} : i\in\{1,\ldots,t\}, x^\beta \in Supp(h_i)\}$. If $Ht(H)=x^\eta$, let $H^+=(h_1^+,\ldots,h_t^+)$ be the $t$-uple such that $h_i^+=c_{i\beta}x^\beta$, where $x^\beta x^{\alpha_i}=x^\eta$, i.e. $x^\beta f_{\alpha_i}\in W_\eta$. Given a homogeneous syzygy $M$ of $J$, we say that $H$ is a {\em lifting} of $M$, or that $M$ {\em lifts to} $H$, if $H^+=M$.
\end{definition}

For the following result we refer to \cite{Mo,Sc}, in particular to Proposition 5.2 of \cite{Mo}.

\begin{corollary}\label{lifting of syzygies}
Every homogeneous syzygy of $J$ lifts to a syzygy of a $J$-marked basis $G$.
\end{corollary}

\begin{proof}
Recall that syzygies of type $(0,\ldots,x^\beta,\ldots,-x^{\beta'},0,\ldots)$ form a system of homogeneous generators of syzygies of $B_J=\{\ldots,x^{\alpha},\ldots,x^{\alpha'},\ldots\}$, where $x^{\beta+\alpha}=x^{\beta'+\alpha'}=lcm(x^{\alpha},x^{\alpha'})$. Thus, apply Theorem \ref{Buchberger}.
\end{proof}

The analogous result of Corollary \ref{lifting of syzygies} for involutive bases immediately holds and it is believable that Janet was aware of that property of involutive polynomials sets.

Until now we have shown that a $J$-marked basis satisfies the characterizing properties of a Gr\"obner basis. In the following result we consider a property that does not characterize Gr\"obner bases, but it is satisfied by Gr\"obner bases. We show that it is satisfied by $J$-marked bases too, by standard arguments. 

\begin{corollary} \label{syzygies generators}
Let $\{M_1,\ldots,M_t\}$ be a set of homogeneous generators of the module of syzygies of $J$. Then, a set $\{K_1,\ldots,K_t\}$ of liftings of the $M_i$'s generates the module of syzygies of $G$.  
\end{corollary}

\begin{proof} 
First, observe that the module of syzygies of $G=\{f_{\alpha_1},\ldots,f_{\alpha_t}\}$ is generated by the syzygies $H=(h_1,\ldots,h_t)$ such that every $h_i=\sum c_{i\beta} x^{\beta}$ is a homogeneous polynomial and every product $h_i f_{\alpha_i}$ has the same degree $m$. Let $H^+$ the syzygy of $J$, as computed in Definition \ref{def:lifting}. Hence, there exist homogeneous polynomials $q_1,\ldots,q_t$ such that $H^+=\sum q_i M_i$. Let $H_1=H-\sum q_i K_i$. By construction we get that $H_{max}(H_1) \prec_m H_{max}(H)$, by Lemmas \ref{reduction1} and \ref{reduction2}. Since $\preceq_m$ is a total order on the finite set $W_m$, the proof is complete.
\end{proof}

\begin{remark}\label{solom0} In the proof of Theorem \ref{Buchberger} we do not use  $V_m$-reductions of all $S$-polynomials $x^\delta f_{\alpha}-x^{\delta'} f_{\alpha'}$ of elements in $G$, but only of those such that either $x^\delta f_{\alpha}$ or $x^{\delta'} f_{\alpha'}$ belongs to some $V_m$. Moreover, we can consider the analogous property to that of the improved Buchberger algorithm that only considers $S$-polynomials corresponding to a set of generators for the syzygies of $J$. Thus we can improve Corollary~\ref{characterization} and say that, under the same hypotheses: 
$$I \in \cM(J) \Longleftrightarrow  \forall m\leq m_0, \  \dim_K I_m=\dim_K J_m \Longleftrightarrow  \forall m\leq m_0, \  \dim_K I_m\leq \dim_K J_m
$$
where $m_0$ is the maximum degree of generators of syzygies of $J$. 
Hence, to prove that $\dim_K I_m=\dim_K J_m $ for some $m$ it is sufficient that the $V_m$-reductions of the $S$-polynomials of degree $\leq m$ are null.
\end{remark}

\begin{example} Let $J=(z^2,zy,zx,y^2)\subset K[x,y,z]$, where $x<y<z$ and consider a $J$-marked set $G=\{f_{z^2},f_{zy}, f_{zx}, f_{y^2}\}$. In order to check whether $G$ is a $J$-marked basis it is sufficient to verify if the polynomials  $S(f_{z^2},f_{zy})$, $S(f_{z^2},f_{zx})$, $S(f_{z^2},f_{y^2})$, $S(f_{zy},f_{zx})$ and $S(f_{zy},f_{y^2})$ have $V_m$-reductions null, but it is not necessary to controll $S(f_{zx},f_{y^2})$ because $yxf_{zy}$ is the element of $V_3$ with head term $zy^2x$.
\end{example}

\begin{example}\label{noGBasis} 
Let $J=(z^3,z^2y,zy^2,y^5)_{\geq 4}$ be a strongly stable ideal in $K[x,y,z]$, with $x<y<z$, and $G=B_J\cup \{f\}\setminus \{zy^2x\}$ a $J$-marked set, where $f=zy^2x-y^4-z^2x^2$ with $Ht(f)=zy^2x$. We can verify that $G$ is a $J$-marked basis using the Buchberger-like criterion proved in Theorem~\ref{Buchberger}. Indeed, the $S$-polynomials non involving $f$ vanish and all the $S$-polynomials involving $f$ are multiple of either $z\cdot(y^4+z^2x^2)$ or $y\cdot(y^4+z^2x^2)$. Since the terms $y^4\cdot z$, $y^4\cdot y$, $z^2x^2\cdot z$, $ z^2x^2\cdot y$ belong to $V_5$, all the $S$-polynomials have $V_m$-reductions null. Notice also that, in this case, $\stackrel{G}\longrightarrow$ is not Noetherian because, although the $V_7$-reduction of $z^2y^2x^3$ is $0$, being $x^2\cdot z^2y^2x \in V_7$ (while $zxf \notin V_7$), a different choice of reduction gives the loop:
$$z^2y^2x^3 \stackrel{f}\longrightarrow zy^4x^2 +z^3x^4  \stackrel{ z^3x^2}\longrightarrow  zy^4x^2  \stackrel{f}\longrightarrow y^6x+z^2y^2x^3 \stackrel{ y^5}\longrightarrow z^2y^2x^3.$$
Morover, $G$ is not a Gr\"obner basis with respect to any term order $\prec$. Indeed, $zy^2x^2\succ y^4x$ and $zy^2x^2 \succ z^2x^3$ would be in contradiction with the equality  $(zy^2x^2)^2 =z^2x^3\cdot y^4x$.
\end{example}


\section{$J$-marked families as affine schemes}

In this section {\em $J$ is always supposed strongly stable}, so that we can use all results described in the previous sections for $J$-marked bases. 

Here we provide the construction of an affine scheme whose points correspond, one to one, to the ideals of the $J$-marked family $\cM(J)$. Recall that $\cM(J)$ is the family of all homogeneous ideals $I$ such that $\cN(J)$ is a basis for $S/I$ as a $K$-vector space, hence $\cM(J)$ contains all homogeneous ideals for which $J$ is the initial ideal with respect to a fixed term order. We generalize to any strongly stable ideal $J$ an approach already proposed in literature in case $J$ is considered an initial ideal (e.g. \cite{CF,FR,LR,Ro,RT}). 

For every $x^\alpha \in B_J$, let $F_\alpha:=x^\alpha - \sum C_{\alpha\gamma} x^\gamma$, where $x^\gamma$ belongs to $\cN(J)_{\vert \alpha \vert}$ and the $C_{\alpha\gamma}$'s are new variables. Let $C$ be the set of such new variables and $N:=\vert C\vert$. The set $\mathcal G$ of all the polynomials $F_\alpha$ becomes a $J$-marked set letting $Ht(F_\alpha)=x^\alpha$. From $\mathcal G$ we can obtain the $J$-marked basis of every ideal $I\in \cM(J)$ specializing in a unique way the variables $C$ in $K^N$, since  every ideal $I\in \cM(J)$ has a unique $J$-marked basis (Remark \ref{rem:unique} and Corollary \ref{cor:ultimo}). But not every specialization gives rise to an ideal of $\cM(J)$. 

Let $\mathcal V_m$ be the analogous for $\mathcal G$ of $V_m$ for any $G$.  
Let $H_{\alpha\alpha'}$ be the $\mathcal V_m$-reductions of the $S$-polynomials $S(F_\alpha,F_{\alpha'})$ of elements of $\mathcal G$ and extract their coefficients that are polynomials in $K[C]$. We will denote by $\mathfrak R$ the ideal of $K[C]$ generated by these coefficients. Let $\mathfrak R'$ be the ideal of $K[C]$ obtained in the same way of $\mathfrak R$ but only considering $S$-polynomials $S(F_\alpha,F_{\alpha'})=x^\delta F_\alpha-x^{\delta'}F_{\alpha'}$ such that $x^\delta F_\alpha$ is minimal among those with head term $x^{\delta+\alpha}$.

\begin{theorem}\label{affine scheme}
There is a one to one correspondence between the ideals of $\cM(J)$ and the points  of the affine scheme  in $K^N$ defined by the ideal $\mathfrak R$. Moreover,  $\mathfrak R'=\mathfrak R$. 
\end{theorem}

\begin{proof}
For the first assertion it is enough to apply Theorem \ref{Buchberger}, observing that a specialization of the variables $C$ in $K^N$ gives rise to a $J$-marked basis if and only if the values chosen for the variables $C$ form a point of $K^N$ on which all polynomials of the ideal $\mathfrak R$ vanish. 

For the second assertion, first recall that, by Remark \ref{solom0}, every $S$-polynomial $x^\delta F_{\alpha}-x^{\delta'} F_{\alpha'}$ can be written as the sum $(x^\delta F_{\alpha}-x^{\delta''} F_{\alpha''})+(x^{\delta''} F_{\alpha''}-x^{\delta'} F_{\alpha'})$ of two $S$-polynomials, where $x^{\delta''} f_{\alpha''}$ belongs to $V_m$. Note that, considering the variables $C$ as parameters, the support of $x^\delta F_{\alpha}-x^{\delta'} F_{\alpha'}$ is contained in the union of the supports of $x^\delta F_{\alpha}-x^{\delta''} F_{\alpha''}$ and of $x^{\delta''} F_{\alpha''}-x^{\delta'} F_{\alpha'}$. In particular, the coefficients in $x^\delta F_{\alpha}-x^{\delta'} F_{\alpha'}$, i.e. the generators of $\mathfrak R$, are combinations of the coefficients in $(x^\delta F_{\alpha}-x^{\delta''} F_{\alpha''})+(x^{\delta''} F_{\alpha''}-x^{\delta'} F_{\alpha'})$, i.e. of the generators of $\mathfrak R'$.
\end{proof}

Now, by exploiting ideas of \cite{LR}, we show how to obtain $\mathfrak R$ in a different way, using the rank of some matrices. 

By Corollary \ref{characterization}, a specialization $C \rightarrow c\in K^N$ trasforms $\mathcal G$ in a $J$-basis $G$ if and only if $dim_K (G)_m = dim_K J_m$, for every degree $m$. Thus, for each $m$, consider the matrix $A_m$ whose columns correspond to the terms of degree $m$ in $S=K[x_1,\ldots,x_n]$ and whose rows contain the coefficients of the terms in every polynomial of degree $m$ of type $x^\delta F_\alpha$. Hence, every entry of the matrix $A_m$ is $1$, $0$ or one of the variables $C$. Let $\mathfrak A$ be the ideal of $K[C]$ generated by the minors of order $dim_K J_m + 1$ of $A_m$, for every $m$.  

\begin{lemma}\label{idealeU}
The ideal $\mathfrak A$ is equal to the ideal $\mathfrak R'$.
\end{lemma}

\begin{proof}
Let $a_m=dim_K J_m$. We consider in $A_m$ the $a_m \times a_m$ submatrix $\bar A_m$ whose columns correspond to the terms $x^\beta x^\alpha$ in $J_m$ and whose rows are given  by the polynomials $x^\beta F_\alpha$ that are minimal with respect to the partial order $>_m$. Up to a permutation of rows and columns, this submatrix is upper-triangular with 1 on the main diagonal because $x^\beta F_\alpha$ is minimal with respect to the partial order $>_m$ and because of Lemma \ref{reduction1}. We may also assume that the submatrix $\bar A_m$ corresponds to the first $a_m$ rows and columns in $A_m$. Then the ideal $\mathfrak{A}$ is generated by the determinants of the $a_m+1\times a_m+1$ sub-matrices containing $\bar A_m$.  Moreover the Gaussian row-reduction of $A_m$ with respect to the first $a_m$ rows is nothing else than the $\mathcal{V}_m$-reduction  of the $S$-polynomials  of the special type considered defining $\mathfrak{R}'$, because the first $a_m$ rows are made of the co!
 efficients of the polynomials of $\mathcal{V}_m$.
\end{proof}

The result of Lemma \ref{idealeU} shows that the construction of the ideal $\mathfrak R$ does not depend on the procedure of reduction. Now, we can give the following definition.

\begin{definition}
The affine scheme defined by the ideal $\mathfrak R=\mathfrak R'=\mathfrak A$ is called {\em $J$-marked scheme}.
\end{definition}

We will denote a $J$-marked scheme and a $J$-marked family by the same symbol $\cM(J)$ because we can identify every ideal $I$ with the corresponding specialization of the variables $C$, by the parameterization of the $J$-marked family on the $J$-marked scheme. We point out that the one-to-one correspondence between the $J$-marked family $\cM(J)$ and the set of projective schemes defined by the ideals of $\cM(J)$ is analogous to the identification of the points of a Hilbert scheme with the projective schemes these points represent.

\begin{remark}
A given homogeneous ideal $I$ belongs to $\cM(J)$ if and only $I$ has the same Hilbert function as $J$ and the affine scheme defined by the ideal of $K[C]$ generated by $\mathfrak R$ and by the coefficients of the $\mathcal V_m$-reductions of the generators of $I$ is not empty. 
Indeed, the ideal $I$ belongs to $\cM(J)$ if and only if it has the same Hilbert function of $J$ and there exists a specialization $\bar C$ in the $J$-marked scheme defined by $\mathfrak R$ such that every generator of $I$ belongs to the ideal $(\bar{\mathcal G})$ generated by the polynomials of $\mathcal G$ evaluated on $\bar C$. The generators of $I$ belong to $(\bar{\mathcal G})$ if and only if their $\mathcal V_m$-reductions evaluated on $\bar C$ become zero.
\end{remark} 

\begin{theorem}\label{homogeneous}
The $J$-marked scheme is homogeneous with respect to a non-standard grading $\lambda$ of $K[C]$ over the group $\mathbb Z^{n+1}$ given by $\lambda(C_{\alpha \gamma})=\alpha-\gamma$. 
\end{theorem} 

\begin{proof}
To prove that the $J$-marked scheme is $\lambda$-homogeneous it is sufficient to show that every minor of $A_m$ is $\lambda$-homogeneous. Let us denote by $C_{\alpha \alpha}$ the coefficient ($=1$) of $x^\alpha$ in every polynomial $F_\alpha$: we can apply also to the \lq\lq symbol\rq\rq\ $C_{\alpha \alpha}$ the  definition of $\lambda$-degree of the variables $C_{\alpha \gamma}$,  because $\alpha-\alpha=0$ is indeed the $\lambda$-degree of the constant $1$. In this way, the  entry in the row $x^\beta F_\alpha$ and in the column $x^\delta$  is  $\pm C_{\alpha \gamma}$ if  $x^{\delta}=x^{\beta}x^{\gamma}$ and is 0 otherwise. 

Let us consider the minor of order $s$ determined in the matrix $A_m$ by the $s$ rows corresponding to $x^{\beta_i}F_{\alpha_i}$ and by the $s$ columns corresponding to $x^{\delta_{j_i}}$, $i=1,\ldots,s$. Every monomial that appears in the computation of such a minor is of type $\prod_{i=1}^s C_{\alpha_i \gamma_{j_i}}$ with  $x^{\delta_{j_i}}=x^{\beta_i}x^{\gamma_{j_i}}$. Then its degree is:
$$\sum_{i=1}^s \left( \alpha_i -\gamma_{j_i}\right) =\sum_{i=1}^s \left( \alpha_i -\delta_{j_i}+\beta_i\right) =\sum_{i=1}^s \left( \alpha_i +\beta_{i}\right) -\sum_{i=1}^s\delta_{j_i}$$
which only depends on the minor.
\end{proof}

Let $\prec$ be a term order and $\St(J,\prec)$ a so-called Gr\"obner stratum \cite{LR}, i.e. the affine scheme that parameterizes all the homogeneous ideals with initial ideal $J$ with respect to $\prec$. We can obtain $\St(J,\prec)$ as the section of $\cM(J)$ by the linear subspace $L$ determined by the ideal $\left(C_{\alpha \gamma} \ : \ x^\alpha \prec x^\gamma \right)\subset K[C]$. In particular, if $m_0$ is defined as in Remark \ref{solom0} and, for every $m\leq m_0$, $J_m$ is a $\prec$-segment, i.e. it is generated by the highest $\dim_K J_m$ monomials with respect to $\prec$, then $\St(J,\preceq)$ and $\cM(J)$ are the same affine scheme. In fact we can obtain both schemes using the same construction. Actually, for some strongly stable ideals $J$ we can find a suitable term ordering such that $\St(J,\prec)= \cM(J)$, but there are cases in which $\bigcup_{\prec}\St(J,\prec)$ is strictly contained in $\cM(J)$ (see the Appendix).

The existence of a term order such that $\cM(J)=\St(J,\preceq)$ has interesting consequences on the geometrical features of the affine scheme $\cM(J)$. In fact the $\lambda$-grading on $K[C]$ is positive if and only if such a term ordering exists and, in this case, we can isomorphically project $\cM(J)$ to the Zariski tangent space at the origin (see \cite{FR}). As a consequence of this projection we can prove, for instance, that the affine scheme $\cM(J)$ is  connected and that it is isomorphic to an affine space, provided the origin is a smooth point. If for a given ideal $J$ such a term ordering does not exist, then in general we cannot embed $\cM(J)$ in the Zariski tangent space at the origin (see the Appendix). However  we do not know examples of Borel ideals $J$ such that either $\cM(J)$ has more than one connected component or $J$ is smooth and $\cM(J)$ is not rational.

\vskip 2mm
Denote by $reg(I)$ the Castelnuovo-Mumford regularity of a homogeneous ideal $I$. 

\begin{proposition} \label{flat}
A $J$-marked family $\cM(J)$ is flat at the origin. In particular, for every ideal $I$ in $\cM(J)$, we get $reg(J)\geq reg(I)$. 
\end{proposition}

\begin{proof}
Analogously to what is suggested in \cite{BM} and by referring to \cite[Corollary, section 3, part I]{A}, we know that $\cM(J)$ is a flat family at $J$, i.e. at the point $C=0$, if and only if every syzygy of $J$ lifts to a syzygy among the polynomials of $\mathcal G$ or, equivalently, the restrictions to $C=0$ of the syzygies of $\mathcal G$ generate the $S$-module of syzygies of $J$. By Corollary \ref{lifting of syzygies} we know that every syzygy of $J$ lifts to a syzygy of $G$, for every specialization of $C$ in the affine scheme defined by the ideal $\mathfrak R$. And this is true thanks to Theorem \ref{Buchberger} that allows also to lift a syzygy of $J$ to a syzygy of $\mathcal G$ over the ring $({K[C]}/{\mathfrak R})[x_0,\ldots,x_n]$. So, the first assertion holds.

For the second assertion, it is enough to recall that the Castelnuovo-Mumford regularity is upper semicontinous in flat families \cite[Theorem 12.8, Chapter III]{H} and that in our case the syzygies of $J$ lift to syzygies of $G$ for every specialization of the variables $C$ in the $J$-marked scheme, i.e., for every ideal $I$ of $\cM(J)$, not only in some neighborhood of $J$. 
\end{proof}

\section*{Appendix: an explicit computation}

Let $J$ be the strongly stable ideal $(z^4, z^3 y, z^2 y^2, z y^3, z^3 x, z^2 y x, z y^2 x, y^5)$ in $K[x,y,z]$ (where $z>y>x$ and $ch(K)=0$), already considered in Example \ref{noGBasis}. Note that for every term order we can find in degree $4$ a monomial in $J$ lower than a monomial in $\cN(J)$, because $zy^2x \succ z^2x^2$ and $zy^2x\succ y^4$ would be in contradiction with the equality $(zy^2x)^2 =z^2x^2\cdot y^4$. Hence, $J_4$ is not a segment (in the usual meaning) with respect to any term order.

The affine scheme $\cM(J)$ can be embedded as a locally closed subscheme in the Hilbert scheme of 8 points in the projective plane (see \cite{LR2}), which is irreducible smooth of dimension 16, and contains all the Gr\"obner strata $\St(J,\prec)$, for every $\prec$, and also some more point, for instance the one corresponding to the ideal $I$ of Example \ref{noGBasis}. 

Let $\mathcal G=\{F_1,\ldots,F_8\}\subset K[z,y,x,c_1,\ldots,c_{64}]$ where the polynomials $F_i$ are
$$F_1=z^4+c_1 x^2 z^2+c_2 y^4+c_3 x^2 y z + c_4 x y^3+c_5 x^3 z + c_6 x^2 y^2 + c_7 x^3 y + c_8 x^4,$$
$$F_2=z^3 y+c_9 x^2 z^2+c_{10} y^4+c_{11} x^2yz+c_{12}xy^3+c_{13} x^3 z+c_{14} x^2 y^2+c_{15} x^3 y+ c_{16} x^4,$$ 
$$F_3=z^2 y^2+c_{17} x^2 z^2+c_{18} y^4+c_{19} x^2 y z+c_{20} x y^3+c_{21} x^3 z+c_{22} x^2 y^2+c_{23} x^3 y+c_{24} x^4,$$ 
$$F_4= z y^3+c_{25} x^2 z^2+c_{26} y^4+c_{27} x^2 y z+c_{28} x y^3+c_{29} x^3 z+c_{30} x^2 y^2+c_{31} x^3 y+c_{32} x^4,$$
$$F_5=z^3x+c_{33} x^2 z^2+c_{34} y^4+c_{35} x^2 y z+c_{36} x y^3+c_{37} x^3 z+c_{38} x^2 y^2+c_{39} x^3y+c_{40} x^4,$$ 
$$F_6=z^2 y x+c_{41} x^2 z^2+c_{42} y^4+c_{43} x^2 y z+c_{44} x y^3+c_{45} x^3 z+c_{46} x^2 y^2+c_{47} x^3 y+c_{48} x^4,$$ 
$$F_7=z y^2 x+c_{49} x^2 z^2+c_{50} y_4+c_{51} x^2 y z+c_{52} x y^3+c_{53} x^3 z+c_{54} x^2 y^2+c_{55} x^3 y+c_{56} x^4,$$ 
$$F_8=y^5+c_{57} x^3 z^2 + c_{58} xy^4+c_{59} x^3 y z+c_{60} x^2 y^3+c_{61} x^4 z+c_{62} x^3 y^2+c_{63} x^4 y+c_{64} x^5.$$ 
By Maple 12 we compute the ideal $\mathfrak R'$ and the following ideal $I(T)$ that defines the Zariski tangent space $T$ to $\cM(J)$ at the origin; note that $T$ has dimension 16:
$$I(T)=
(c_{64}, c_{63}, c_{61}, c_{56}, c_{55}, c_{53}, c_{48}, c_{47}, c_{46}, c_{45}, c_{44}, c_{40}, c_{39}, c_{38}, c_{37}, c_{36}, c_{32}, c_{31}, c_{30}, c_{29},$$ $$c_{28}-c_{54}, c_{27},c_{26}-c_{52}, c_{25}, c_{24}, c_{23}, c_{22}, c_{21}, c_{20}, c_{19}, c_{18}, c_{17}, c_{16}, c_{15}, c_{14}, c_{13},$$ $$c_{12}, c_{11}, c_{10},c_9, c_8, c_7, c_6, c_5, c_4, c_3, c_2, c_1).$$
In the ideal $\mathfrak R'$ we eliminate several variables of type $C$ by applying \cite[Theorem 5.4]{LR2} and by substituting variables that appear only in the linear part of some polynomials of $\mathfrak R'$. We obtain that $\cM(J)$ can be isomorphically projected on a linear space $T'\simeq \Af^{19}$ containing $T$. In this embedding, $\cM(J)$ is the complete intersection of the following three hypersurfaces in $\Af^{19}$ of degrees $4$, $4$ and $8$, respectively:

$G_1=c_{41}^2 c_{49} c_{50}+c_{41} c_{49} c_{50} c_{51}+c_{41} c_{50}^2 c_{57}+c_{42} c_{49} c_{50} c_{57}+c_{43} c_{49}^2 c_{50}+c_{49}c_{50}^2 c_{59}$
$+c_{49} c_{50} c_{51}^2 +c_{50}^2 c_{51} c_{57}+c_{50}^2 c_{57} c_{58}-c_{41} c_{49} c_{52}-c_{49} c_{50} c_{53}-c_{49} c_{51} c_{52}$
$-2 c_{50} c_{52} c_{57}+
+c_{33} c_{49}-c_{41}^2 +c_{41} c_{51}-c_{42} c_{57}-c_{43} c_{49}+c_{49} c_{54}-c_{53},$

$G_2=c_{41} c_{42} c_{49} c_{50}+c_{42} c_{49} c_{50} c_{51}+c_{42} c_{49} c_{50} c_{58}+c_{42} c_{50}^2 c_{57}+c_{43} c_{49} c_{50}^2 +c_{50}^3 c_{59}+$ $c_{50}^2 c_{51}^2$
$+c_{50}^2 c_{51} c_{58}+c_{50}^2 c_{58}^2 -c_{42} c_{49} c_{52}-c_{44} c_{49} c_{50}-c_{50}^2 c_{53}-c_{50}^2 c_{60}-2 c_{50} c_{51} c_{52}$
$-2 c_{50} c_{52} c_{58}+c_{34} c_{49}-c_{41} c_{42}+c_{42} c_{51}-c_{42} c_{58}-c_{43} c_{50}+2 c_{50} c_{54}+c_{52}^2 +c_{44},$

$G_3\ = \ -c_{41}^3 c_{49}^3 c_{50}^2\ -c_{41}^2 c_{49}^3 c_{50}^2 c_{51}\ +c_{41}^2 c_{49}^3 c_{50}^2 c_{58} \ - \ 2 c_{41}^2 c_{49}^2 c_{50}^3 c_{57}+$
$c_{41} c_{42}^2 c_{49}^5\ +2 c_{41} c_{49}^3 c_{50}^2 c_{51} c_{58}\ +c_{41} c_{49}^3 c_{50}^2 c_{58}^2\ -2 c_{41} c_{49}^2 c_{50}^3 c_{51} c_{57}\ -c_{41} c_{49} c_{50}^4 c_{57}^2\ +c_{42}^2 c_{49}^5 c_{51}+$
$c_{42}^2 c_{49}^5 c_{58} \ -\ c_{49}^3 c_{50}^2 c_{51} c_{58}^2 \ -\ c_{49}^3 c_{50}^2 c_{58}^3\ +\ 2 c_{49}^2 c_{50}^3 c_{51} c_{57} c_{58}\ +$
$2 c_{49}^2 c_{50}^3 c_{57}c_{58}^2\ -\ c_{49} c_{50}^4 c_{51} c_{57}^2\ -\ c_{49} c_{50}^4 c_{57}^2 c_{58}\ +\ 2 c_{41}^2 c_{49}^3 c_{50} c_{52}\ -\ 2 c_{41} c_{42} c_{49}^4 c_{52}$ 
$\ - \ 4 c_{41}c_{49}^3 c_{50} c_{52} c_{58} + \ 4 c_{41} c_{49}^2 c_{50}^2 c_{52} c_{57} - \ 2 c_{42} c_{44} c_{49}^5\ - \ 2 c_{42} c_{49}^4 c_{50} c_{60}$
$-\ 2 c_{42} c_{49}^4 c_{51} c_{52}\ +\ 2 c_{49}^3 c_{50} c_{52} c_{58}^2\ -$ 
$4 c_{49}^2 c_{50}^2 c_{52} c_{57} c_{58}+2 c_{49} c_{50}^3 c_{52} c_{57}^2-2 c_{33} c_{41} c_{49}^3 c_{50} + 2 c_{33}c_{49}^3 c_{50} c_{58} - 2 c_{33} c_{49}^2 c_{50}^2 c_{57} + 2 c_{34} c_{41} c_{49}^4 -2 c_{34} c_{49}^4 c_{58}+$
$2 c_{34} c_{49}^3 c_{50} c_{57}+4 c_{41}^3c_{49}^2 c_{50}$
$-c_{41}^2 c_{42} c_{49}^3 -2 c_{41}^2 c_{49}^2 c_{50} c_{51}$ $-4 c_{41}^2 c_{49}^2 c_{50} c_{58}$
$+5 c_{41}^2 c_{49} c_{50}^2 c_{57}+3 c_{41} c_{42}c_{49}^3 c_{51}
+4 c_{41} c_{42} c_{49}^2 c_{50} c_{57}+3 c_{41} c_{43} c_{49}^3 c_{50}+c_{41} c_{49}^3 c_{52}^2 +c_{41} c_{49}^2 c_{50}^2 c_{59}
+c_{41} c_{49}c_{50}^2 c_{51} c_{57}+2 c_{41} c_{50}^3 c_{57}^2 +c_{42} c_{43} c_{49}^4 +2 c_{42} c_{49}^4 c_{54}+3 c_{42} c_{49}^3 c_{50} c_{59} + \ c_{42} c_{49}^3 c_{51}^2 - \ c_{42} c_{49}^3 c_{51} c_{58} + \ c_{42} c_{49}^3 c_{58}^2 -$ $2 c_{42} c_{49}^2 c_{50} c_{51} c_{57}-4 c_{42} c_{49}^2 c_{50} c_{57} c_{58}+2 c_{42} c_{49} c_{50}^2 c_{57}^2 -3 c_{43} c_{49}^3 c_{50} c_{58}+3 c_{43} c_{49}^2 c_{50}^2 c_{57}+2 c_{44} c_{49}^4 c_{52}+2 c_{49}^3 c_{50} c_{52} c_{60}+c_{49}^3 c_{51} c_{52}^2-c_{49}^3 c_{52}^2 c_{58}-c_{49}^2 c_{50}^2 c_{58} c_{59}-c_{49}^2 c_{50} c_{51}^3 -2 c_{49}^2 c_{50} c_{51}^2 c_{58}+c_{49} c_{50}^3 c_{57} c_{59}-c_{49} c_{50}^2c_{51}^2 c_{57}
-5 c_{49} c_{50}^2 c_{51} c_{57} c_{58}-3 c_{49} c_{50}^2 c_{57} c_{58}^2 +2 c_{50}^3 c_{51} c_{57}^2 +2 c_{50}^3 c_{57}^2 c_{58}-c_{41}c_{49}^2 c_{52}+c_{41} c_{44} c_{49}^3 +c_{41} c_{49}^2 c_{50} c_{60}+c_{41} c_{49}^2 c_{51} c_{52}+c_{41} c_{49}^2 c_{52} c_{58}-5 c_{41} c_{49} c_{50} c_{52} c_{57}-2 c_{42} c_{49}^3 c_{53}+c_{42} c_{49}^3 c_{60}+c_{42} c_{49}^2 c_{52} c_{57}
+c_{43} c_{49}^3 c_{52}-2 c_{44}c_{49}^3 c_{51}-c_{44} c_{49}^3 c_{58}-2 c_{49}^3 c_{52} c_{54}
+ \ c_{49}^2 c_{50} c_{51} c_{60} - \ c_{49}^2 c_{50} c_{52} c_{59}+ \ 2 c_{49}^2 c_{51}^2 c_{52}+ \
c_{49}^2 c_{51} c_{52} c_{58}+2 c_{49} c_{50}^2 c_{57} c_{60}+5 c_{49} c_{50} c_{51} c_{52} c_{57}+6 c_{49} c_{50} c_{52} c_{57} c_{58}-4 c_{50}^2 c_{52}c_{57}^2 +c_{33} c_{41} c_{49}^2 -2 c_{33} c_{49}^2 c_{51}-c_{33} c_{49}^2 c_{58}+c_{33} c_{49} c_{50} c_{57}-3 c_{34} c_{49}^2 c_{57}
-c_{35} c_{49}^3-2 c_{41}^3 c_{49}+2 c_{41}^2 c_{49} c_{51}+2 c_{41}^2 c_{49} c_{58}-3 c_{41}^2 c_{50} c_{57}-c_{41} c_{42} c_{49} c_{57}-2 c_{41} c_{43} c_{49}^2+c_{41} c_{49}^2 c_{54}-2 c_{41} c_{49} c_{50} c_{59}-3 c_{41} c_{49} c_{51}^2 +c_{41} c_{50} c_{51} c_{57}-c_{42} c_{49}^2 c_{59}-c_{42} c_{49} c_{51} c_{57}+3 c_{42} c_{49} c_{57} c_{58}-3 c_{42} c_{50} c_{57}^2 +2 c_{43} c_{49}^2 c_{58}-2 c_{43} c_{49} c_{50} c_{57}-c_{49}^2 c_{50} c_{62}-2 c_{49}^2 c_{51} c_{54}-c_{49}^2 c_{54} c_{58}-2 c_{49} c_{50} c_{51} c_{59}-c_{49} c_{50} c_{54} c_{57}-c_{49} c_{51}^3 +c_{49}c_{51}^2 c_{58}-2 c_{49} c_{52}^2 c_{57}-c_{50}^2 c_{57} c_{59}-c_{50} c_{51}^2 c_{57}-c_{41} c_{49} c_{60}+c_{41} c_{52} c_{57}-c_{44} c_{49} c_{57}+2 c_{49} c_{50} c_{61}+4 c_{49} c_{51} c_{53}-c_{49} c_{51} c_{60}+c_{49} c_{52} c_{59}-c_{50} c_{53} c_{57}+c_{51} c_{52} c_{57}+c_{33} c_{57}+c_{41} c_{59}+c_{49} c_{62}-c_{54} c_{57}-c_{61}.
$
\medskip
Among the generators of the corresponding Jacobian ideal we have the following minors  $D_i$ obtained by computing the derivatives of $G_1,G_2,G_3$ with respect to the sets of variables $A_i$, for $1\leq i\leq 5$:  
\vskip 1mm
$D1 =- (2 c_{49} c_{50}-1)(c_{49} c_{50}-1)(c_{49} c_{50}+1)$, \quad $A_1=\{c_{61},c_{44},c_{53}\}$;

$D_2=-(c_{49}c_{50}+1)(c_{49}c_{50}-1)^2c_{49}$,
\quad $A_2=\{c_{53},c_{44},c_{62}\}$;

$D_3=-c_{50}(2c_{49}c_{50}-1)(c_{49}c_{50}-1)$,
\quad $A_3=\{c_{43},c_{61},c_{53}\}$;

$D_4=c_{49}(c_{49}c_{50}-1)^2(2c_{49}c_{50}-1)$,
\quad $A_4=\{c_{43},c_{61},c_{44}\}$;

$D_5=(c_{49}c_{50}+1)c_{50}^2(2c_{49}c_{50}-1)$,
\quad $A_5=\{c_{53},c_{60},c_{61}\}$.

\vskip 1mm
\noindent The polynomials $D_i$ define the empty set, so that $\cM(J)$ is smooth as we expected and, in particular, $J$ corresponds to a smooth point on $\cM(J)$. Moreover, $\cM(J)$ has dimension $16$ but we claim that it cannot be isomorphically projected on $T$. Indeed, note that we can choose a set of $16$ variables that is complementary to the tangent space and that does not contain the variables $c_{53}, c_{44}, c_{61}$ which occur in the linear parts of the polynomials $G_i$. These variables appear also in other parts of the polynomials and their coefficients are $c_{49}c_{50}+1$, $c_{49}c_{50}-1$ and $2 c_{49} c_{50}-1$, respectively. 
If $\bar c\in \mathbb T$ is a point of the tangent space on which none of the coefficients vanishes, we obtain a unique point of $\cM(J)$ of which $\bar c$ is the projection on $T$.
If $\bar c\in T$ is a general point of the tangent space on which one of these coefficients vanishes, one can see that $\bar c$ is not the projection of any point of $\cM(J)$. Hence, the projection of $\cM(J)$ on $T$ does not coincide with the tangent space $T$, but only with an open set. However, this fact implies that $\cM(J)$ is rational, in particular irreducible.

We point out that the variables $c_{49}$ and $c_{50}$, that appear in the coefficients of the variables $c_{53}, c_{44}, c_{61}$, are the coefficients in the polynomial $F_7$ of the two terms $x^2 z^2$, $y^4$ whose behaviour prevents the ideal $J$ from being a segment. Indeed, in this case the affine scheme $\cM(J)$ is homogeneous with respect to a non-positive grading.


\providecommand{\bysame}{\leavevmode\hbox to3em{\hrulefill}\thinspace}
\providecommand{\MR}{\relax\ifhmode\unskip\space\fi MR }
\providecommand{\MRhref}[2]{%
  \href{http://www.ams.org/mathscinet-getitem?mr=#1}{#2}
}
\providecommand{\href}[2]{#2}

\end{document}